\documentclass[12pt]{amsart}
\usepackage{amssymb}
\usepackage{amsmath, amscd}
\usepackage{amsthm}
\usepackage{xcolor}
\usepackage{ulem}

\newtheorem{theorem}{Theorem}[section]

\newtheorem{lemma}[theorem]{Lemma}

\newtheorem{corollary}[theorem]{Corollary}
\theoremstyle{definition}

\newtheorem{definition}[theorem]{Definition}

\def\mda#1{\downarrow \rlap{$\vcenter{\hbox{$\scriptstyle#1$}}$}}

\usepackage[utf8]{inputenc}
\usepackage{amssymb}
\usepackage{amsmath, amscd}
\usepackage{amsthm}
\usepackage{xcolor}

\begin{document}

\author[A.R. Chekhlov]{Andrey R. Chekhlov}
\address{Department of Mathematics and Mechanics, Tomsk State University, 634050 Tomsk, Russia}
\email{cheklov@math.tsu.ru; a.r.che@yandex.ru}
\author[P.V. Danchev]{Peter V. Danchev}
\address{Institute of Mathematics and Informatics, Bulgarian Academy of Sciences, 1113 Sofia, Bulgaria}
\email{danchev@math.bas.bg; pvdanchev@yahoo.com}
\author[P.W. Keef]{Patrick W. Keef}
\address{Department of Mathematics, Whitman College, Walla Walla, WA 99362, USA}
\email{keef@whitman.edu}

\title[Generalizing the Bassian and co-Bassian Properties for Abelian Groups]{Generalizations of the Bassian and co-Bassian \\ Properties for Abelian Groups}

\maketitle

\begin{abstract} Trying to finalize in some way the present subject, this paper targets to generalize substantially the notions of Bassian and co-Bassian groups by introducing the so-called {\it finitely (co-)Bassian groups}, {\it semi (co-)Bassian groups}, {\it fully generalized (co-)Bassian groups}, {\it absolutely generalized (co-)Bassian groups} and establishing their crucial properties and characterizations. In fact, some of the concepts give nothing new by coinciding in the reduced case with the well-known (co-)Bassian property. However, in some of the definitions, the situation is slightly more complicated and we obtain a few new and interesting things by showing that the extensions of the Bassian and co-Bassian properties are totally distinct each other.
\end{abstract}

\medskip
\medskip
\medskip

\centerline{(Dedicated to the anniversary of the {\bf100}th birthday of L\'aszl\'o Fuchs)}

\vskip2.0pc

\section{Definitions and Motivations}

All groups considered in the current paper will be Abelian and additively written. Our notations and terminology, as usual, will mainly agree with the classical books \cite{F1} and \cite{F}, respectively. For instance, if $G$ is a group, the letter $T=T(G)$ will denote its maximal torsion subgroup with $p$-primary component $T_p=T_p(G)$. Likewise, if $x\in G$, then the symbol $|x|_p$ will denote its $p$-height as computed in $G$. Also, for any integer $n\geq 1$, the notation $\mathbb Z(p^n)$ will stand for the traditional cyclic group of order $p^n$.

\medskip

The following concepts are central to our investigations (see \cite{CDG1}, \cite{CDG2}, \cite{DK}, \cite{CDK} as well): a group $G$ is said to be {\it Bassian} if it cannot be embedded in a proper homomorphic image of itself or, equivalently, if the existence of an injective homomorphism $G \to G/N$, for some subgroup $N$ of $G$, forces that $N = \{0\}$.

\medskip

Generally, if this injection implies that $N$ is a direct summand of $G$, then $G$ is named {\it generalized Bassian} and, in a more general form, if $N$ is essential in a direct summand of $G$, then $G$ is termed {\it semi-generalized Bassian}.

\medskip

Restated slightly, the Main Theorem of \cite{CDG1} asserts that {\it a group $G$ is Bassian if, and only if, it has finite torsion-free rank and, for all primes $p$, $T_p$ is also finite}. Thus, it is pretty clear that the standard quasi-cyclic $p$-group $\mathbb{Z}(p^{\infty})$ is definitely {\it not} Bassian.

\medskip

Unfortunately, the generalized Bassian groups were not completely characterized in \cite{CDG2}, but it was shown there that if $G$ is generalized Bassian, then $pT_p$ must be finite for all primes $p$ (\cite[Lemma 3.1]{CDG2}) as well as it was established in \cite{DK} that such a group $G$ has finite torsion-free rank.

\medskip

We now can state the following four new specifications of these Bassian properties.

\begin{definition}\label{finite} A group $G$ is called {\it finitely Bassian} provided that, for any subgroup $N$ of $G$, whenever there is an injection $G\to G/N$, then $N$ is finite.
\end{definition}

It is pretty clear that Bassian groups are finitely Bassian, and also it is obvious that torsion-free finitely Bassian groups are Bassian. Generally, we shall curiously illustrate in what follows even that the classes of finitely Bassian and Bassian groups do coincide in the reduced case.

\medskip

We just now recall that an arbitrary subgroup $E$ of a group $G$ is {\it essential} in $G$ if, for any non-zero subgroup $S$ of $G$, the intersection between $E$ and $S$ is also non-zero. It is an elementary exercise to see that every subgroup will be essential as a subgroup of itself and thus, in particular, the zero subgroup $\{0\}$ is essential in itself too.

\medskip

So, we introduce the following.

\begin{definition}\label{semi} A group $G$ is called {\it semi Bassian} provided that, for any subgroup $N$ of $G$, whenever there is an injection $G\to G/N$, then either $N=\{0\}$ or $N$ is essential in $G$.
\end{definition}

Thus, under these circumstances, each Bassian group is semi Bassian, but we shall demonstrate in the sequel that the converse manifestly fails.

\medskip

We recall that a subgroup $F$ of a group $G$ is {\it fully invariant} if $\phi(F)\subseteq F$ for all endomorphism $\phi$ of $G$. So, we arrive at the following.

\begin{definition}\label{fully} A group $G$ is called {\it fully generalized Bassian} provided that, for any subgroup $N$ of $G$, whenever there is an injection $G\to G/N$, then $N$ is a fully invariant direct summand of $G$.
\end{definition}

It is quite apparent that these groups are situated between Bassian groups and generalized Bassian groups. We, however, will show in the sequel that the converse manifestly fails, i.e., these three classes are independent each other, thus motivating the study in-depth of this interesting class of groups.

\medskip

We recall that a subgroup $A$ of a group $G$ is its {\it absolute direct summand} if, for any $A$-high subgroup $V$ of $G$, it must be that $G=A\oplus V$. By definition, such a group $V$ is maximal with respect to $V\cap A=\{0\}$.

\begin{definition}\label{absolutely} A group $G$ is called {\it absolutely generalized Bassian} provided that, for any subgroup $H$ of $G$, whenever there is an injection $G\to G/H$, then $H$ is an absolute direct summand of $G$.
\end{definition}

It is rather evident that these so-defined groups are situated between Bassian groups and generalized Bassian groups. We, however, will show in the sequel that in two of situations the converse manifestly fails, namely that the classes of semi-Bassian and absolutely generalized Bassian groups are independent each other, thus motivating their study in-depth. Unfortunately, the fully generalized Bassian groups are precisely the Bassian ones.

\medskip

On the other vein, in \cite{K} was introduced two important classes of groups like this: a group $G$ is said to be {\it co-Bassian} or, respectively, {\it generalized co-Bassian} if, for all subgroups $H\leq G$, whenever there is an injection $\phi: G\to G/H$ either $\phi(G)=G/H$ or, respectively, the group $\phi(G)$ is a direct summand of the factor-group $G/H$. Fortunately, these groups were completely described in terms of invariants.

\medskip

More generally, in \cite{CDK1} was considered and examined the case of groups $G$ where $\phi(G)$ is essential in a direct summand of the quotient-group $G/H$, calling them {\it semi-generalized co-Bassian} and succeeded to characterize all such $p$-groups and groups of finite torsion-free rank.

\medskip

Nevertheless, it is rather interesting to know what happens when the image $\phi(G)$ of the group $G$ is either of finite index in $G/H$, or is an essential subgroup of $G/H$, or is a fully invariant direct summand of $G/H$, or is an absolute direct summand of $G/H$, respectively.

\medskip

Precisely, we state the following four new notions.

\begin{definition}\label{finitelyco} A group $G$ is called {\it finitely co-Bassian} provided that, for any subgroup $H$ of $G$, whenever there is an injection $\phi: G\to G/H$, then $\phi(G)$ has finite index in $G/H$.
\end{definition}

\begin{definition}\label{semico} A group $G$ is called {\it semi co-Bassian} provided that, for any subgroup $H$ of $G$, whenever there is an injection $\phi: G\to G/H$, then $\phi(G)$ is essential in $G/H$.
\end{definition}

We, however, will show in the sequel that these two notions curiously do coincide and will classify them accordingly.

\begin{definition}\label{fullyco} A group $G$ is called {\it fully generalized co-Bassian} provided that, for any subgroup $H$ of $G$, whenever there is an injection $\phi: G\to G/H$, then $\phi(G)$ is a fully invariant direct summand of $G/H$.
\end{definition}

\begin{definition}\label{absolutelyco} A group $G$ is called {\it absolutely generalized co-Bassian} provided that, for any subgroup $H$ of $G$, whenever there is an injection $\phi: G\to G/H$, then $\phi(G)$ is an absolute direct summand of $G/H$.
\end{definition}

And so, our basic motivation to explore significantly these groups is that the class of generalized Bassian groups is far from being completely characterized (cf. \cite{CDG1}, \cite{CDG2} and \cite{DK}) as well as that the class of generalized co-Bassian groups is only numerically characterized (see \cite{K}), and thus the given above specifications of the property "direct summand" in the corresponding definitions are worthy of consideration and investigation.

\medskip

That is why, our further work is subsequently organized in the following manner: in the next two sections, we state and prove our major results, which motivate us to write this research. The first one contains four subsections pertaining to generalizations of the Bassian property in which the main achievements are Theorems~\ref{major1}, \ref{major2}, \ref{major3} and \ref{major4}, respectively, as well as the second one also contains three subsections pertaining to generalizations of the co-Bassian property in which the chief achievements are Theorems~\ref{major5}, \ref{major6} and \ref{major7}, respectively.

\section{Generalizing the Bassian Property}

For the convenience of the reader, we include the following.

\begin{lemma}\cite [Lemma~2.4]{K}\label{flag}
Suppose $G$ has finite (torsion-free) rank, $N$ is a subgroup of $G$, $\pi:G\to \overline G:=G/N$ is the usual epimorphism and $\phi: G\to \overline G$ is an injection. Then, $N\subseteq T$, so that $\overline T=T/N$ is the torsion subgroup of $\overline G$. In addition, $\pi$ induces an isomorphism $\pi':G/T\cong \overline G/\overline T$ and $\phi$ induces an embedding $\phi':G/T\to \overline G/\overline T$. Finally, if $G/T$ is divisible, then $\phi'$ is an isomorphism.
\end{lemma}

\begin{proof} Let $S$ be the torsion subgroup of $\overline G$; clearly, $\phi(T)=\phi(G)\cap S$ and $\pi(T)\subseteq S$. So, $\phi$ and $\pi$ induce injective and surjective homomorphisms $$\phi', \pi':G/T\to \overline G/S,$$ respectively.

Since both $G/T$ and $\overline G/S$ are torsion-free of finite rank, it easily follows that $$\pi':G/T\to \overline G/S$$ is an isomorphism. And since $$\pi'([N+T]/T)=[\pi(N)+S]/S=\{0\},$$ we have $[N+T]/T=\{0\}$, i.e., $N\subseteq T$, so that $S=\overline T$. Finally, if $G/T$ is divisible, then an easy verification shows that $\phi'$ is an isomorphism, as required.
\end{proof}

We will make frequent use of the following construction.

\begin{lemma}\label{basic} Suppose $G$ is a group that has subgroups $N\leq M$ such that $M/N$ is divisible and there is an injection $M\to M/N$. Then, there is a splitting $G/N\cong (M/N)\oplus (G/M)$ and an injective homomorphism $\phi:G\to G/N$.
\end{lemma}

\begin{proof} Since $M/N$ is divisible, the splitting $$G/N\cong (M/N)\oplus (G/M)$$ is immediate. Note that the injection $M\to M/N$ will extend to a homomorphism $\gamma:G\to M/N$. For each $x\in G$, defining $\phi(x)=(\gamma(x), x+M)$ can readily be seen to be injective.
\end{proof}

Here we completely describe the afore-defined three kinds of generalized Bassian groups as follows.

\subsection{Finitely Bassian groups}

We begin this section with the following quite surprising assertion by showing that, actually, there is nothing new in Definition~\ref{semi} concerning the reduced case. In other words, the classes of reduced semi-Bassian groups and reduced Bassian groups curiously do coincide.

However, the case of divisible groups is also of some interest being non-trivial. For example, consider $G={\mathbb Z}(2^\infty)\oplus {\mathbb Z}(3^\infty)$. If $N\leq G$ and $\phi:G\to G/N$ is injective, it plainly follows that $N$ is finite and even, in fact, cyclic. So, this $G$ is finitely Bassian, as expected.

\medskip

Concretely, we offer the following.

\begin{theorem}\label{major1} The group $G$ is finitely Bassian if, and only if, $G=A\oplus C$, where $A$ is a divisible hull of a finite group and $C$ is a Bassian group.
\end{theorem}

\begin{proof} Suppose first that $G=A\oplus C$ is as above, $N$ is a subgroup of $G$ and $$\phi:G\to G/N=: \overline G$$ is an injection. So, by Lemma~\ref{flag}, it must be that $N\leq T$. If $\pi:G\to \overline G$ is the canonical epimorphism, then for every prime $p$, the maps $\pi, \phi$ induce, respectively, surjective and injective homomorphisms $T_p\to \overline T_p$. This automatically implies that $T_p$ and $\overline T_p$ have the same $p$-rank. Now, if $N_p$ was infinite, we could conclude that it is not reduced, which would imply that $\overline T_p\cong T_p/N_p$ would have $p$-rank strictly less than the $p$-rank of $T_p$. Since this cannot happen, each $N_p$ is finite.

For any prime $p$ such that $A_p=\{0\}$, $T_p$ is finite, so that it is Bassian. Therefore, for such $p$, $N_p=\{0\}$. And since this is true for all but finitely many primes, we can conclude that $N$ is finite.

\medskip

Conversely, suppose that $G$ is finitely Bassian. Let $A$ be the maximal divisible subgroup of $T$, so that $G=A\oplus C$ for some subgroup $C$. If $A$ decomposed into an infinite number of non-zero terms, it would follow that there was an infinite subgroup $N\leq A$ such that $A\cong A/N$ (for example, $N$ could be the socle of $A$). This would immediately imply an isomorphism $$\phi:G=A\oplus C\to (A/N)\oplus C,$$ so that $G$ is not finitely Bassian.  Consequently, $A$ must be a divisible hull of a finite subgroup, for instance, its socle.

Let $S$ be the torsion subgroup of $C$. If $p$ is a prime, we want to show that $S_p$ is finite. We first claim that it is bounded: if this failed, then $S_p$ would have a countable unbounded reduced subgroup $M$. This $M$, in turn, will have a basic subgroup $N$ such that $M/N$ has countably infinite rank. Since there is clearly an injection $M\to M/N$, thanks to Lemma~\ref{basic}, there is an injection $$\phi:G\to G/N\cong (M/N)\oplus (G/M).$$ And since $N$ is infinite, this contradicts that $G$ is finitely Bassian.

Therefore, $S_p$ is bounded. If it were infinite, then for some integer $k$ there would be a decomposition $$G=(\oplus_{i<\omega} \mathbb Z(p^k)_i)\oplus G',$$ where each $\mathbb Z(p^k)_i$ is a copy of $\mathbb Z(p^k)$. Clearly, if $N=\oplus_{i<\omega} \mathbb Z(p^k)_{2i}$, then $N$ is infinite and $G/N\cong G$, contradicting that $G$ is finitely Bassian.

Thus, each $S_p$ is finite, so we must show that $G$ has finite rank. If this failed, we could find a free subgroup $F\leq G$ of countably infinite rank, and a subgroup $N\leq F$ such that $F/N$ is divisible of countably infinite rank. Since there is clearly an injection $F\to F/N$, using Lemma~\ref{basic}, there is an injection $$\phi:G\to (F/N)\oplus (G/F).$$ And since $N$ is infinite, this contradicts that $G$ is finitely Bassian, as needed.
\end{proof}

\subsection{Semi Bassian groups}

We characterize the semi Bassian groups in the next result.

\begin{theorem}\label{major2} The group $G$ is semi Bassian if, and only if, either $G\cong \mathbb{Z}(p^{\infty})$
or $G$ is a Bassian group.
\end{theorem}

\begin{proof}  Suppose first that $G$ is semi Bassian. If $T$ is not reduced, then $G=G'\oplus Z$, where $Z\cong \mathbb{Z}(p^{\infty})$. Since $G\cong G/Z[p]$, we could conclude that $Z[p]$ is essential in $G$. This clearly implies that $G'=\{0\}$, so that we find $G=Z\cong \mathbb{Z}(p^{\infty})$. So, we may assume $T$ is reduced.

If $p$ is a prime, we want to show that $T_p$ is finite. We first claim that it is bounded: if this failed, then $T_p$ would have a countable unbounded reduced subgroup $M$. This $M$, in turn, will have a basic subgroup $N$ such that $M/N$ has countable rank. Since there is an obvious injection $M\to M/N$, appealing to Lemma~\ref{basic}, there is an injection $$\phi:G\to G/N\cong (M/N)\oplus (G/M).$$ And since $N$ is not essential in $M$, it is apparently not essential in $G$ too; this surely contradicts that $G$ is semi Bassian.

Therefore, $T_p$ is bounded. If it were infinite, then, for some integer $k$, there would be a decomposition $$G=(\oplus_{i<\omega} \mathbb Z(p^k)_i)\oplus G',$$ where each $\mathbb Z(p^k)_i$ is a copy of $\mathbb Z(p^k)$. Evidently, if $N=\oplus_{i<\omega} \mathbb Z(p^k)_{2i}$, then $N$ is clearly not essential in $G$ and $G/N\cong G$, contradicting that $G$ is finitely Bassian.

Thus, each $T_p$ is finite, so to show $G$ is Bassian we must prove that $G$ has finite (torsion-free) rank. However, if this failed, we could find a free subgroup $F\leq G$ of countably infinite rank, and a subgroup $N\leq F$ such that $F/N$ is divisible of countably infinite rank. Since there is plainly an injection $F\to F/N$, using Lemma~\ref{basic}, there is an injection $$\phi:G\to (F/N)\oplus (G/F).$$ And since, again, $N$ is not essential in $M$, it is directly not essential in $G$ as well. This again contradicts that $G$ is semi Bassian, as needed.

Since a quasi-cyclic or a Bassian group is immediately semi-Bassian, the converse is automatic.
\end{proof}

It is worthy of noticing that if we consider the following modified version of Definition~\ref{semi}, there will be nothing new and interesting. In fact, let us call a group $G$ {\it almost Bassian} provided that, for any non-zero subgroup $N$ of $G$, whenever there is an injection $G\to G/N$, then $N$ is essential in $G$.

\medskip

In fact, $G$ will be ``almost Bassian" exactly when it is ``semi Bassian": indeed, if $P(\phi)$ is the statement ``$\phi$ is an injection", $Q(N)$ is the statement ``$N=\{0\}$" and $R(N)$ is the statement ``$N$ is essential in $G$", then the statement ``$G$ is semi Bassian" can be written symbolically as
$$
               \forall (N\leq G)\forall (\phi:G\to G/N)(P(\phi)\Rightarrow (Q(N)\lor R(N))
$$
and ``$G$ is almost Bassian" can be written symbolically as    		
$$
               \forall (N\leq G)\forall (\phi:G\to G/N)((P(\phi)\land \lnot Q(N))\Rightarrow R(N)).
$$
But, it is readily to see that the stuff inside the quantifiers is tautologically equivalent.

\subsection{Fully generalized Bassian groups}

We start in this section with the following quite surprising assertion by demonstrating that, actually, there is nothing new in Definition~\ref{fully}. In other words, the classes of fully generalized Bassian groups and Bassian groups curiously do coincide.

To show that, we first recall that a group $G$ is {\it B+E} if it is isomorphic to the direct sum $B\oplus E$, where $B$ is a Bassian group and $E$ is an elementary group.

We now have the following helpful assertion.

\begin{corollary}(\cite[Corollary 3.1]{DK}) Every generalized Bassian group is {\rm {B+E}}.
\end{corollary}

We are now in a position to establish the promised above equivalence.

\begin{theorem}\label{major3} The group $G$ is fully generalized Bassian if, and only if, $G$ is a Bassian group.
\end{theorem}

\begin{proof} Certainly, by definition, any Bassian group is fully generalized Bassian, so assume $G$ is fully generalized Bassian. It follows that $G$ is generalized Bassian, and hence B+E as stated in the previous statement. Hence, write $G=B\oplus E$ with $B$ Bassian and $E$ elementary.

If we can show, for each $p$, that $E_p$ has finite $p$-rank, then it will follow that $G$ is also Bassian, as required. So, by way of contradiction, suppose $p$ is some prime such that $E_p$ has infinite $p$-rank. It readily follows that $G=G'\oplus E_p$ for some subgroup $G'\leq G$. Let $N$ be a non-zero cyclic direct summand of $E_p$. It must be that $E_p\cong E_p/N$, so that $G\cong G/N$. However, since $N$ is obviously not fully invariant in $G$, we deduce that $G$ is {\it not} fully generalized Bassian, giving our contradiction, as pursued.
\end{proof}

\subsection{Absolutely generalized Bassian groups}

First of all, we are prepared to use the following well-known statement.

\begin{lemma}\cite[\S10, Exersize 9]{F1}\label{Fuchs} The direct summand $B$ of an arbitrary group $A$ is its absolute direct summand if, and only if, $B$ is either a divisible group or the quotient $A/B$ is a torsion group such that the $p$-component of which is equal to zero at multiplying by $p^k$ if $B\setminus pB$ has an element of order $p^k$.
\end{lemma}

Now, we are ready to prove here the following assertion.

\begin{theorem}\label{major4} The group $G$ is absolute generalized Bassian if, and only if, exactly one of the following two conditions holds:

(1) $G$ is a torsion group with elementary or finite $T_p$ for each prime $p$.

(2) $G$ is a Bassian non-torsion group.
\end{theorem}

\begin{proof} The sufficiency being an almost straightforward observation, we shall attack the necessity. So, assume $G$ is absolute generalized Bassian, but not Bassian; thus, we need to show that $G$ is a torsion group satisfying the required property (1).

To that goal, since $G$ is absolute generalized Bassian, it is generalized Bassian, and hence B+E as already noticed above. Therefore, for any prime $p$, $T_p$ is the direct sum of an elementary and a finite group. Suppose $p$ is any prime such that $T_p$ (and hence $G$) has a maximal elementary direct summand $E_p$ of infinite $p$-rank; since $G$ is not Bassian, such a $p$ exists. If $G=E_p\oplus G'$ and $N$ is a cyclic direct summand of $E_p$, then it is also a direct summand of $G$. Again, there is an isomorphism $E_p\cong E_p/N$ which extends to an isomorphism $G\to G/N$. Therefore, since $G$ is absolutely generalized Bassian, $N$ must be an absolute direct summand of $G$. By usage of Lemma~\ref{Fuchs}, this implies that $G'$ is a torsion group with no $p$-torsion. However, since the above argument works for all such primes $p$, it clearly implies that $G$ is of the form presented in (1), as needed.
\end{proof}

So, we arrived at the interesting fact that the class of absolute generalized Bassian groups does {\it not} coincide with the class of Bassian groups and properly lies between the two classes of Bassian groups and generalized Bassian groups.

\section{Generalizing the co-Bassian Property}

Here we completely describe the afore-defined two kinds of generalized co-Bassian groups as follows.

\subsection{Finitely and semi co-Bassian groups}

In \cite {K}, a $p$-group $T_p$ was said to have {\it generalized finite $p$-rank} if there are elements of $\omega \cup \{\infty\}$, $\alpha_1< \alpha_2<\cdots <\alpha_n$ and cardinals $\rho_i$ ($i=1, 2, \dots, n$) such that
$$T_p\cong \mathbb Z(p^{\alpha_1})^{(\rho_1)}\oplus \cdots \oplus \mathbb Z(p^{\alpha_n})^{(\rho_n)},$$ where $\rho_j$ is finite whenever $j>1$.

We now recollect the following two statements.

\begin{theorem}[\cite{K}, Theorem~2.5]\label{genco-Bassian} The group $G$ is generalized co-Bassian if, and only if, one of the following two conditions holds:

(1) $G$ is divisible;

(2) $G/T$ is divisible of finite rank and, for each prime $p$, $T_p$ has generalized finite $p$-rank.
\end{theorem}

\begin{theorem}[\cite{K}, Theorem~2.6]\label{co-Bassian} The group $G$ is co-Bassian if, and only if, $G/T$ is divisible of finite rank and, for each prime $p$, $T_p$ has finite $p$-rank.
\end{theorem}

We pause for a simple but useful observation.

\begin{lemma}\label{sum}
Suppose $N\leq G$ and $\overline G=G/N$ decomposes as $(\oplus_{i<\omega} A_i)\oplus G'$, where each $A_i$ is a copy of a single non-zero group, and $\phi: G\to \overline G$ is an injective homomorphism. Then, $G$ will be neither finitely co-Bassian nor semi co-Bassian.
\end{lemma}

\begin{proof}
Let $\gamma:\overline G\to \overline G$ be the identity on $G'$, and map each $A_i$ isomorphically onto $A_{2i}$.
It, thus, follows at once that $\phi':=\gamma\circ \phi$ is injective and hence a simple verification guarantees that $$\phi'(G)\subseteq (\oplus_{i<\omega} A_{2i})\oplus G'$$ is neither essential in $\overline G$, nor is the quotient $\overline G/\phi'(G)$ finite, as wanted.
\end{proof}

This brings us to our characterization of finitely co-Bassian and semi co-Bassian groups, which classes curiously coincide.

\begin{theorem}\label{major5} For a group $G$, the following are equivalent:

(a) $G$ is finitely co-Bassian;

(b) $G$ is semi co-Bassian;

(c) $G$ has finite torsion-free rank and each $T_p$ has finite $p$-rank.
\end{theorem}

\begin{proof} We will show that if $G$ is of the form specified in (c), then it satisfies both (a) and (b); and conversely, if it is {\it not} of that form, then both (a) and (b) fail.

Suppose first that $G$ satisfies (c), and let $N\leq G$, $\pi: G\to \overline G:= G/N$ be the canonical epimorphism and $\phi:G\to \overline G$ be an injective homomorphism. It follows from Lemma~\ref{flag} that $N\leq T$ and $T/N=:\overline T$ is the torsion subgroup of $\overline G$. In addition, $\phi$ and $\pi$ restrict to injective and surjective homomorphisms $T\to \overline T$, respectively. However, by Theorem~\ref{co-Bassian}, $T$ is co-Bassian. Therefore, $\overline T=\phi(T)$.

Consider now the commutative diagram with short-exact rows, whose columns are induced by the injection $\phi$:
$$
    \begin{matrix}  \{0\}&\to& T&\to & G& \to & G/T & \to & \{0\} \cr
                    &&\mda {\phi_T}&&\mda \phi &&\mda {\phi'}&& \cr
                    \{0\}&\to& \overline T &\to & \overline G & \to & \overline G/\overline T & \to & \{0\} \cr
    \end{matrix}
$$
Again referring to Lemma~\ref{flag}, we detect that $\phi'$ will be injective. So, since $G/T$ and $\overline G/\overline T$ are torsion-free groups of the same finite rank and $\phi_T$ is onto, it elementarily follows that $\phi(G)$ is essential in $\overline G$, proving (b).

In addition, the above commutative diagram implies an isomorphism
$$G/\phi(G)\cong (\overline G/\overline T)/\phi'(G/T).$$
And on the other hand, since $\phi'$ is injective, and with the aid Lemma~\ref{flag} we represent $G/T\cong \overline G/\overline T$, one can conclude that $$\phi'(G/T)\cong G/T\cong \overline G/\overline T.$$ Recall that, whenever $X$ is a torsion-free group of finite rank and $Y$ is a subgroup of $X$ that is isomorphic to $X$, then we can infer that $X/Y$ is finite. Therefore, $(\overline G/\overline T)/\phi'(G/T)$ is finite, so that $G/\phi(G)$ is finite, which proves (a).

\medskip

Conversely, suppose $G$ does not satisfy (c); we want to show that (a) and (b) fail, as well. First, if $p$ is a prime, $n<\omega$ and the $n$th Ulm invariant $(p^n T)[p]/(p^{n+1}T)[p]$ is infinite, then Lemma~\ref{sum} (with $N=\{0\}$, $\pi=1_G=\phi$ and $A_i\cong \mathbb Z(p^{n+1})$) implies that $G$ does not satisfy either (a) or (b). So,
we may assume this does not hold for any prime $p$.

It follows that if $T_p$ does not have finite rank, then it has an countable unbounded reduced subgroup $M$. If $N$ is a basic subgroup of $M$ such that $M/N$ has countably infinite rank, then it is clear that $M$ embeds in $M/N$. So, with Lemma~\ref{basic} at hand, $G$ embeds in $$G/N\cong (M/B)\oplus (G/M),$$ so that Lemma~\ref{sum} guarantees that $G$ does not satisfy (a) or (b), as pursued.

Thus, we may assume that $T_p$ has finite rank for all primes $p$. It will suffice, therefore, to show that it $G$  has finite rank. If this failed, then $G$ would have a free subgroup $F$ of countable rank, and $F$ would have a subgroup $N$ such that $F/N$ is a torsion-free divisible group of countably infinite rank. As it is again clear that $F$ embeds in $F/N$, by Lemma~\ref{basic}, $G$ embeds in $$G/N\cong (F/N)\oplus (G/F),$$ so that by Lemma~\ref{sum}, $G$ does not satisfy (a) or (b), as desired.
\end{proof}

\subsection{Fully generalized co-Bassian groups}

We are now planning to establish the following result.

\begin{theorem}\label{major6} The group $G$ is fully generalized co-Bassian if, and only if, it is co-Bassian, i.e., $G/T$ is divisible of finite rank and, for each prime $p$, $T_p$ has finite $p$-rank.
\end{theorem}

\begin{proof} It is elementary that if $G$ is co-Bassian, then it is fully generalized co-Bassian.

Conversely, suppose $G$ is fully generalized co-Bassian. In particular, it immediately follows that $G$ is generalized co-Bassian, so that Theorem~\ref{genco-Bassian} is applicable.

Note that if $G=G'\oplus (\oplus_{i<\omega} A_i)$, where each $A_i$ is a copy of a single non-zero group $A$, then $G$ cannot be fully generalized co-Bassian: indeed, for $N=\{0\}$, there is an obvious injection $\phi: G\to G\cong G/N$ such that $$\phi(G)=G'\oplus (\oplus_{i<\omega} A_{2i})\subseteq G.$$ But $\phi(G)$ is certainly not fully invariant in $G\cong G/N$, as asked for.

So, if Theorem~\ref{genco-Bassian}(a) applies, i.e., $G$ is divisible, then it routinely follows that each divisible $T_p$ will have finite $p$-rank and $G/T$ will be divisible of finite rank, completing the proof.

On the other hand, suppose Theorem~\ref{genco-Bassian}(b) applies. We already know that $G/T$ will be divisible of finite rank. And if $p$ is a prime, then we know that $T_p$ must have generalized finite $p$-rank. For such a prime, if $T_p$ did not actually have finite rank, then we could conclude that $G$ has a direct summand of the form
${\mathbb Z}(p^\alpha)^{(\omega)}$, where $\alpha$ is in $\omega\cup \{\infty\}$. But, again by the above observation, this implies that $G$ is {\it not} fully generalized co-Bassian, as expected.
\end{proof}

\subsection{Absolutely generalized co-Bassian groups}

We now have all the ingredients necessary to attack the truthfulness of the following affirmation.

\begin{theorem}\label{major7} The group $G$ is absolutely generalized co-Bassian if, and only if, it is generalized co-Bassian, i.e., $G$ is either divisible, or $G/T$ is divisible of finite rank and each $T_p$ has generalized finite rank.
\end{theorem}

\begin{proof} Certainly, if $G$ is absolutely generalized co-Bassian, then it must be generalized co-Bassian. So, assume $G$ is generalized co-Bassian. So, if $$\pi, \phi:G\to G/N=: \overline G$$ are surjective and injective homomorphisms, we need to verify that $\phi(G)$ is an absolute direct summand of $\overline G$.

Clearly, if $G$ is divisible, then so is $\phi(G)$, so that it is an absolute direct summand, as required.
So, we may assume that $G/T$ is divisible of finite rank and each $T_p$ has generalized finite rank.
By Lemma~\ref{flag}, $\phi$ induces an isomorphism $\phi':G/T\cong \overline G/\overline T$ on the corresponding  quotients. In addition, $N\subseteq T$ and, consequently, $$\overline T=T/N\cong \oplus_p T_p/N_p\cong \oplus_p \overline T_p$$
will be the torsion subgroup of $\overline G$.

Since $G$ is generalized co-Bassian, it must be that $\overline G= \phi(G)\oplus C$. Furthermore, as $\phi$ induces an isomorphism $\phi':G/T\cong \overline G/\overline T$ on the corresponding  quotients, it follows that $C\subseteq  \overline T$, so that $C$ must be a torsion group.

To prove $\phi(G)$ is an absolute direct summand of $\overline G$, let $p$ be some prime. Since $T_p$ has generalized finite rank, there is a decomposition $T_p= H_p\oplus F_p$, where either $H_p$ is divisible and $F_p=\{0\}$, or there is an $k<\omega$ such that $H_p\cong \mathbb Z(p^k)^{(\gamma_p)}$ and $F_p$ has finite $p$-rank and no direct summand isomorphic to $\mathbb Z(p^j)$ for any $j\leq k$.

If, however, $T_p=H_p$ is divisible, then there are no elements in $\phi(G)\setminus p\phi(G)$ of order $p^j$ for any  $j<\omega$. So, suppose $T_p=H_p\oplus F_p$ and $k<\omega$ is as above. It follows that the maps $\pi, \phi$ restrict to surjective and injective homomorphism $p^k T_p\to p^k \overline T_p$, respectively. Consequently, since $p^k T_p$ is a $p$-group of finite rank, so that it is co-Bassian, we have $\phi: p^k \phi(T_p)=p^k \overline T_p$. This implies that $p^k C_p=\{0\}$, so that by virtue of Lemma~\ref{Fuchs}, the image $\phi(G)$ is an absolute direct summand of $\overline G$, as stated.
\end{proof}

\medskip
\medskip

\noindent {\bf Funding:} The work of the first-named author, A.R. Chekhlov, is supported by the Ministry of Science and Higher Education of Russia under agreement No. 075-02-2023-943. The work of the second-named author, P.V. Danchev, is partially supported by the Junta de Andaluc\'ia, Grant FQM 264.

\end{document}